 \documentclass[10pt]{amsart}

\usepackage{amssymb}
\usepackage{bbm}
\usepackage{anysize}
\marginsize{1in}{1in}{1in}{1in}

\usepackage{pdfsync}
\usepackage{color}
\newtheorem{tm}{Theorem}[section]

\newtheorem{defin}[tm]{Definition} 
\newtheorem{exmp}[tm]{Example}
\newtheorem{coro}[tm]{Corollary}
\newtheorem{lem}[tm]{Lemma}
\newtheorem{assumption}[tm]{Assumption}
\newtheorem{rk}[tm]{Remark}

\numberwithin{equation}{section}

\begin{document}


\title{Stability of stochastic differential equation driven by time-changed L\'evy noise}

\author{ERKAN NANE}
\address{Department of Mathematics and Statistics,
Auburn University,
Auburn, AL 36849 USA}
\email{ezn@auburn.edu}

\author{YINAN NI}
\address{Department of Mathematics and Statistics,
Auburn University,
Auburn, AL 36849 USA}
\email{yzn0005@auburn.edu}



\begin{abstract}
This paper studies stabilities of stochastic differential equation (SDE) driven by time-changed L\'evy noise in both probability and moment sense. This provides more flexibility in modeling schemes in application areas including physics, biology, engineering, finance and hydrology. Necessary conditions for solution of time-changed SDE to be stable in different senses will be established. Connection between stability of solution to time-changed SDE and that to corresponding original SDE will be disclosed. Examples related to different stabilities will be given. We study SDEs with time-changed L\'evy noise, where the  time-change processes  are inverse of general L\'evy subordinators. These results are important improvements of the results in Wu \cite{qwa}.

\end{abstract}


\maketitle


\section{Introduction}
It has been a long time since stochastic differential equations (SDEs) started being applied in various areas, including biology \cite{such}, physics \cite{dope}, engineering \cite{sob}, finance \cite{jijo}. SDEs are taken as important tools in modeling and simulating real phenomena, the stability of SDEs has been studied widely by mathematicians in different senses, such as stochastically stable, stochastically asymptotically stable, moment exponentially stable, almost surely stable, mean square polynomial stable, see \cite{apsia, foondun,sia,qwb}. A systematic introduction of stabilities is provided by Mao in \cite{maotext}.

During last few decades, time-changed SDEs attracted lots of attention and became one of the most active areas in stochastic analysis and many applied areas of science. Their probability density functions provide solutions to fractional Fokker-Planck equations of different kinds, see \cite{maer,erni}, which are also very important in modeling and describing phenomena in applied areas, see \cite{jber}.

In \cite{keib} Kobayashi discussed relationship between time-changed SDEs
\begin{equation}\label{equk}
\begin{aligned}
&dX(t)=f(E_t,X(t-))dE_t+g(E_t,X(t-))dZ_{E_t},\\
& X(0)=x_0,\\
\end{aligned}
\end{equation}
and the corresponding non-time-changed SDEs
\begin{equation}\label{equk2}
\begin{aligned}
&dY(t)=f(t,Y(t-))dt+g(t,Y(t-))dZ_t,\\
& Y(0)=x_0,\\
\end{aligned}
\end{equation}
where $Z_t$ is an $\mathcal{F}_t$-semimartingale and $E_t$ is an inverse of a right continuous with left limit (RCLL) nondecreasing process $\{D(t),t \geq 0\}$: if a process $Y(t)$ satisfies SDE \eqref{equk2}, then  $X(t):=Y(E_t)$ satisfies the time-changed SDE \eqref{equk}; if a process $X(t)$ satisfies the time-changed SDE \eqref{equk}, then $Y(t):=X(D(t))$ satisfies SDE \eqref{equk2}.

Kobayashi also studied It\^o formula driven by time-changed SDE which is provided under certain conditions as below,
\begin{equation}
\begin{aligned}
f(X_t)-f(x_0)=&\int_0^tf'(X_{s-})A_sds+\int_0^{E_t}f'(X_{D(s-)-})F_{D(s-)}ds\\
&+\int_0^{E_t}f'(X_{D(s-)-})G_{D(s-)}dZ_s\\
&+\frac{1}{2}\int_0^{E_t}f''(X_{D(s-)-})\{G_{D(s-)}\}^2d[Z,Z]_s^c\\
&+\sum_{0<s\leq t}\{f(X_s)-f(X_{s-})-f'(X_{s-})\Delta X_s\},
\end{aligned}
\end{equation}
where $f: \mathbb{R}\rightarrow \mathbb{R}$ is a $C^2$ function.

In light of time-changed It\^o formula, recent paper \cite{qwa} analyzes the SDE driven by time-changed Brownian motion
\begin{equation}\label{equw}
\begin{aligned}
&dX(t)=k(t,E_t,X(t-))dt+f(t,E_t,X(t-))dE_t+g(t,E_t,X(t-))dB_{E_t},\\
& X(0)=x_0,\\
\end{aligned}
\end{equation}
where $E_t$ is specified as an inverse of a stable subordinator of index $\beta$ in $(0,1)$, and discusses the stability of solution to above SDE in probability sense, including stochastically stable, stochastically asymptotically stable and globally stochastically asymptotically stable.

Main result of this paper is to provide necessary conditions for solutions of SDEs driven by time-changed L\'evy noise to be stable not only in probability sense but also in moment sense. Our results improve the results of \cite{qwa} in two respects. Firstly, we study SDEs with time-changed L\'evy noise. Secondly, we work with time-change processes that are inverse of general L\'evy subordinators.

In the remaining parts of this paper, further needed concepts and related background will be given in the preliminary section.  In the main result section, necessary conditions for solution of time-changed SDEs to be stable in different senses will be given. Connections between stability of solution to time-changed SDE and that to corresponding original SDE will be disclosed and some examples will be given. Last section will show proofs of theorems mentioned in main result section.

\section{Preliminaries}

Let $(\Omega, \mathcal{F},(\mathcal{F}_t)_{t\geq 0},P)$ be a filtered probability space satisfying usual hypotheses of completeness and right continuity. Let $\mathcal{F}_t$-adapted Poisson random measure $N$ defined on $\mathbb{R}_+\times (\mathbb{R}-\{0\})$ with compensator $\tilde{N}$ and intensity measure $\nu$, where $\nu$ is a L\'evy measure such that $\tilde{N}(dt,dy)=N(dt,dy)-\nu(dy)dt$ and $\int_{\mathbb{R}-\{0\}}(|y|^2 \land 1)\nu(dy)<\infty$.

Let $\{D(t),t\geq 0\}$ be a RCLL increasing L\'evy process that is called subordinator starting from 0 with Laplace transform
\begin{equation}
\mathbbm{E}e^{-\lambda D(t)}=e^{-t\phi(\lambda)},
\end{equation}
where Laplace exponent $\phi(\lambda)=\int_0^\infty(1-e^{-\lambda x})\nu(dx)$.

Define its inverse
\begin{equation}\label{invstable}
E_t:=\inf\{ \tau>0: D(\tau)>t\}.
\end{equation}

This paper focuses on different stabilities of the following SDE:
\begin{equation}\label{SDE}
\begin{aligned}
dX(t)&=f(t, E_t, X(t-))dt+k(t, E_t, X(t-))dE_t+g(t, E_t, X(t-))dB_{E_t}\\
&+\int_{|y|<c}h(t, E_t, X(t-),y)\tilde{N}(dE_t,dy),
\end{aligned}
\end{equation}
with $X(0)=x_0$, where $f,k,g,h$ are real-valued functions satisfying the following Lipschitz condition \ref{lip} and assumption \ref{tec}
such that there exists a unique $\mathcal{G}_t=\mathcal{F}_{E_t}$ adapted process $X(t)$ satisfying time changed SDE \eqref{SDE}, see Lemma 4.1 in \cite{keib}.


\begin{assumption}\label{lip}
(Lipschitz condition)
There exists a positive constant K such that
\begin{equation}
\begin{aligned}
&\Big|f(t_1,t_2,x)-f(t_1,t_2,y)\Big|^2+\Big|k(t_1,t_2,x)-k(t_1,t_2,y)\Big|^2+\Big|g(t_1,t_2,x)-g(t_1,t_2,y)\Big|^2\\
&+\int_{|z|<c}\Big|h(t_1,t_2,x,z)-h(t_1,t_2,x,z)\Big|^2\nu(dz)\leq K|x-y|^2,
\end{aligned}
\end{equation}
for all $t_1,t_2\in \mathbb{R}_+$ and $x,y\in \mathbb{R}$.
\end{assumption}

\begin{assumption}\label{tec}
If $X(t)$ is a RCLL and $\mathcal{G}_t$-adapted process, then
\begin{equation}
f(t, E_t, X(t)), k(t, E_t, X(t)), g(t, E_t, X(t)),h(t, E_t, X(t),y)\in \mathcal{L}(\mathcal{G}_t),
\end{equation}
where  $\mathcal{L}(\mathcal{G}_t)$ denotes the class of RCLL and $\mathcal{G}_t$-adapted processes.
\end{assumption}

Next we give definitions of different stabilities of SDE \eqref{SDE}.

\begin{defin}
(1) The trivial solution of the time-changed SDE \eqref{SDE} is said to be stochastically stable or stable in probability if for every pair of $\epsilon \in (0,1)$ and $r>0$, there exists a $\delta=\delta(\epsilon, r)>0$ such that
\begin{equation}
P\{|X(t,x_0)|<r\ for\ all\ t \geq 0\}\geq 1-\epsilon
\end{equation}
whenever $|x_0|<\delta$.\\

(2) The trivial solution of the time-changed SDE \eqref{SDE} is said to be stochastically asymptotically stable if for every $\epsilon \in (0,1)$, there exists a $\delta_0=\delta_0(\epsilon)>0$ such that
\begin{equation}
P\{\lim_{t\rightarrow \infty}X(t,x_0)=0\}\geq 1-\epsilon
\end{equation}
whenever $|x_0|<\delta_0$.\\

(3) The trivial solution of the time-changed SDE \eqref{SDE} is said to be globally stochastically asymptotically stable or stochastically asymptotically stable in the large if it is stochastically stable and for all $x_0 \in \mathbb{R}$
\begin{equation}
P\{\lim_{t\rightarrow \infty}X(t,x_0)=0\}=1.
\end{equation}
\end{defin}

\begin{defin}
(1) The trivial solution of the time-changed SDE \eqref{SDE} is said to be pth moment exponentially stable if
there are positive constants $\lambda$ and C such that
\begin{equation}
E[|X(t)|^p]\leq C|x_0|^p\exp(-\lambda t),\ \forall t\geq 0,\ \forall x_0\in \mathbb{R},\ p>0.
\end{equation}

(2) The trivial solution of the time-changed SDE \eqref{SDE} is said to be pth moment asymptotically stable if
there is a function $v(t):[0,+\infty)\rightarrow [0,\infty)$ decaying to 0 as $t\rightarrow \infty$ and a positive constant C such that
\begin{equation}
E[|X(t)|^p]\leq C|x_0|^p v(t),\ \forall t\geq 0,\ \forall x_0\in \mathbb{R},\ p>0.
\end{equation}
\end{defin}

Let $\mathcal{K}$ denote the family of all nondecreasing functions $\mu:\mathbb{R}_+\rightarrow \mathbb{R}_+$ such that $\mu(r)>0$ for all $r>0$. Also let $S_h=\{x\in \mathbb{R}: |x|<h\}$  and $\bar{S_h}=\{x\in \mathbb{R}: |x| \leq h\}$ for all $h>0$.

\section{Main results}
In this section, time-changed It\^o formula driven by SDE \eqref{SDE} will be given, then necessary conditions for different stabilities will be established, followed by some examples.

The next lemma is a version of the It\^o formula in Corollary 3.4 in \cite{keib}.

\begin{lem}({It\^o formula for time-changed L\'evy noise})\label{itofor}
Let $D(t)$ be a RCLL subordinator and $E_t$ its inverse process as \eqref{invstable}. Define a filtration $\{\mathcal{G}_t\}_{t\geq 0}$ by $\mathcal{G}_t=\mathcal{F}_{E_t}$. Let $X$ be a process defined as following:
\begin{equation}
\begin{aligned}\label{sdelevy}
X(t)&=x_0+\int_0^tf(t, E_t, X(t-))dt+\int_0^tk(t, E_t, X(t-))dE_t+\int_0^tg(t, E_t, X(t-))dB_{E_t}\\
&+\int_0^t\int_{|y|<c}h(t, E_t, X(t-),y)\tilde{N}(dE_t,dy),
\end{aligned}
\end{equation}
where $f,k,g,h$ are measurable functions such that all integrals are defined. Here $c$ is the maximum allowable jump size.\\
Then, for all $F : \mathbb{R}_+\times \mathbb{R}_+\times \mathbb{R}\rightarrow \mathbb{R}$ in $C^{1,1,2}(\mathbb{R}_+\times \mathbb{R}_+\times \mathbb{R},\mathbb{R})$, with probability one,
\begin{equation}
\begin{aligned}\label{itolevy}
F(t, E_t, &X(t))-F(0,0,x_0)=\int_0^t L_1F(s, E_s, X(s-))ds+\int_0^t L_2F(s, E_s, X(s-))dE_s\\
&+\int_0^t\int_{|y|<c}\Big[F(s, E_s, X(s-)+h(s, E_s, X(s-),y))-F(s, E_s, X(s-))\Big]\tilde{N}(dE_s,dy)\\
&+\int_0^t F_x(s, E_s, X(s-))g(s, E_s, X(s-))dB_{E_s},
\end{aligned}
\end{equation}
where
\begin{equation}
\begin{aligned}\label{linearop}
L_1F(t_1,&t_2,x)=F_{t_1}(t_1,t_2,x)+F_{x}(t_1,t_2,x)f(t_1,t_2,x),\\
L_2F(t_1,&t_2,x)=F_{t_2}(t_1,t_2,x)+F_{x}(t_1,t_2,x)k(t_1,t_2,x)+\frac{1}{2}g^2(t_1,t_2,x)F_{xx}(t_1,t_2,x)\\
+&\int_{|y|<c}\Big[F(t_1,t_2,x+h(t_1,t_2,x,y))-F(t_1,t_2,x)-F_x(t_1,t_2,x)h(t_1,t_2,x,y)\Big]\nu(dy).
\end{aligned}
\end{equation}
\end{lem}

\begin{proof}
This proof is a direct application of multidimensional It\^o formula, which is established in Corollary 3.4 in \cite{keib}, to $F(t, E_t, X(t))$ in $C^{1,1,2}(\mathbb{R}_+\times \mathbb{R}_+\times \mathbb{R},\mathbb{R})$. \\
\begin{equation}
\begin{aligned}
F(t, &E_t, X(t))-F(0,0,x_0)=\int_0^t F_{t_1}(s, E_s, X(s-))ds+\int_0^t F_{t_2}(s, E_s, X(s-))dE_s\\
&+\int_0^t F_x(s, E_s, X(s-))\Big[f(s, E_s, X(s-))ds+k(s, E_s, X(s-))dE_s\\
&\hfill+g(s, E_s, X(s-))dB_{E_s}\Big]+\frac{1}{2}\int_0^t F_{xx}(s, E_s, X(s-))g(s, E_s, X(s-))dE_s\\
&+\int_0^t\int_{|y|<c}\Big[F(s, E_s, X(s-)+h(s, E_s, X(s-),y))-F(s, E_s, X(s-))\Big]\tilde{N}(dE_s,dy)\\
\end{aligned}
\end{equation}
\begin{equation*}
\begin{aligned}
&+\int_0^t\int_{|y|<c}\Big[F(s, E_s, X(s-)+h(s, E_s, X(s-),y))-F(s, E_s, X(s-))\\
&\hspace{5.8cm}-F_x(s, E_s, X(s-))h(s, E_s, X(s-),y)\Big]\nu(dy)dE_s\\
\end{aligned}
\end{equation*}
\begin{equation*}
\begin{aligned}
&=\int_0^t L_1F(s, E_s, X(s-))ds+\int_0^t L_2F(s, E_s, X(s-))dE_s\\
&+\int_0^t\int_{|y|<c}\Big[F(s, E_s, X(s-)+h(s, E_s, X(s-),y))-F(s, E_s, X(s-))\Big]\tilde{N}(dE_s,dy)\\
&+\int_0^t F_x(s, E_s, X(s-))g(s, E_s, X(s-))dB_{E_s}.
\end{aligned}
\end{equation*}
\end{proof}

\begin{lem}
Let $D(t)$ be a RCLL subordinator and $E_t$ be its inverse process as in \eqref{invstable}. Define a filtration $\{\mathcal{G}_t\}_{t\geq 0}$ by $\mathcal{G}_t=\mathcal{F}_{E_t}$. Let $\tilde{N}$ be a compensated Poisson measure defined on $\mathbb{R}_+\times (\mathbb{R}-\{0\})$ with intensity measure $\nu$, where $\nu$ is a L\'evy measure such that $\tilde{N}(dt,dy)=N(dt,dy)-\nu(dy)dt$ and $\int_{\mathbb{R}-\{0\}}(|y|^2 \land 1)\nu(dy)<\infty$. Then, for any $A\in\mathcal{B}(\mathbb{R}-\{0\})$ bounded below, time-changed process $\tilde{N}(E_t, A)$ is a martingale.\\
\end{lem}

\begin{proof}
Let $\tau_n=\inf\{t\geq 0; |\tilde{N}(t,A)|\geq n\}$, it is obvious that $\tau_n \rightarrow \infty$ as $n\rightarrow \infty$. Then $|\tilde{N}(\tau_n\land t,A)|\leq n+1$, for all $t\in \mathbb{R}_+$, thus $\tilde{N}(\tau_n\land t,A)$ is a bounded martingale.

By optional stopping theorem, for any $0 \leq s<t$,
\begin{equation}
\mathbb{E}\Big[\tilde{N}(\tau_n \land E_t, A)|\mathcal{G}_s \Big]=\tilde{N}(\tau_n \land E_s, A).
\end{equation}

The right hand side $\tilde{N}(\tau_n \land E_s, A)$ converges to $\tilde{N}(E_s, A)$, as $n \rightarrow \infty$. For the left hand side, we have
\begin{equation}
|\tilde{N}(\tau_n \land E_t, A)|\leq \sup_{0\leq u \leq t}|\tilde{N}(E_u, A)|,
\end{equation}

thus, by H\"older's inequality, Doob's martingale inequality,
\begin{equation}
\begin{aligned}
\mathbb{E}\Big[\Big|\tilde{N}(\tau_n \land E_t, A)\Big| \Big]&\leq\mathbb{E}\Big[ \Big| \sup_{0\leq u \leq t}|\tilde{N}(E_u, A)\Big| \Big]=\mathbb{E}\Big[ \Big| \sup_{0\leq u \leq E_t}|\tilde{N}(u, A)\Big| \Big]\\
&=\int_0^{\infty}\mathbb{E}\Big[ \Big| \sup_{0\leq u \leq \tau}|\tilde{N}(u, A)\Big|\Bigg| \tau=E_t\Big] f_{E_t}(\tau) d\tau\\
&\leq\int_0^{\infty}\mathbb{E}\Big[ \Big| \sup_{0\leq u \leq \tau}|\tilde{N}(u, A)\Big|^2\Bigg| \tau=E_t\Big]^\frac{1}{2} f_{E_t}(\tau) d\tau\\
&\leq \int_0^{\infty}2\mathbb{E}\Big[ \Big| |\tilde{N}(\tau, A)\Big|^2\Bigg| \tau=E_t\Big]^\frac{1}{2} f_{E_t}(\tau) d\tau\\
&=2\int_0^{\infty}[\nu(A)\tau]^\frac{1}{2}f_{E_t}(\tau) d\tau\\
&=2\nu(A)^\frac{1}{2}\mathbb{E}[E_t^\frac{1}{2}].\\
&\leq2\nu(A)^\frac{1}{2}\mathbb{E}[E_t]^\frac{1}{2},
\end{aligned}
\end{equation}
where the last inequality follows from Jensen's inequality.

For any $t\geq 0$ and $x>0$, by Markov's inequality, we have
\begin{equation}
P(E_t>s)\leq P(D(s)<t)=P(e^{-xD(s)}\geq e^{-xt})\leq e^{xt}\mathbbm{E}[e^{-xD(s)}]=e^{xt}e^{-s\phi(x)},
\end{equation}
it follows that
\begin{equation}
\mathbbm{E}[E_t]=\int_0^\infty P(E_t>s)ds=e^{xt}\frac{1}{\phi(x)}<\infty.
\end{equation}

Then, by dominated convergence theorem, we have
\begin{equation}
\mathbb{E}\Big[\tilde{N}(\tau_n \land E_t, A)|\mathcal{G}_s \Big]\rightarrow \mathbb{E}\Big[\tilde{N}( E_t, A)|\mathcal{G}_s \Big],
\end{equation}
as $n\rightarrow \infty$. So
\begin{equation}
\mathbb{E}\Big[\tilde{N}( E_t, A)|\mathcal{G}_s \Big]=\tilde{N}(E_s, A).
\end{equation}

Also, \begin{equation}
\mathbb{E}\Big[|\tilde{N}( E_t, A)|\Big]\leq \mathbb{E}\Big[ \sup_{0\leq u \leq t}|\tilde{N}( E_u, A)|\Big]<\infty,
\end{equation}
thus $\tilde{N}(E_t, A)$ is a martingale.

\end{proof}

\begin{tm}\label{tm1}
Assume that there exists a function $V(t_1,t_2,x)\in C^{1,1,2}(\mathbb{R}_+\times \mathbb{R}_+\times S_h,\mathbb{R})$ with $h\geq 2c$ and $\,\mu\in \mathcal{K}$ such that for all $(t_1,t_2,x) \in \mathbb{R}_+\times \mathbb{R}_+\times S_h$\\

\begin{equation}
\begin{aligned}
&1.\ V(t_1,t_2,0)=0,\\
&2.\ \mu(|x|)\leq V(t_1,t_2,x), \\
&3.\ L_1V(t_1,t_2,x) \leq 0,\\
&4.\ L_2V(t_1,t_2,x) \leq 0,
\end{aligned}
\end{equation}
then the trivial solution of the time-changed SDE \eqref{SDE} is stochastically stable or stable in probability.
\end{tm}

\begin{proof}
See Section \ref{appd}.
\end{proof}

\begin{rk}
Note that $L_1$ and $L_2$ mentioned here and in following theorems are same as these in Lemma \ref{itofor}, c is maximum allowable jump size in \eqref{SDE}.
\end{rk}

\begin{tm}\label{tm2}
Assume that there exists a function $V(t_1,t_2,x)\in C^{1,1,2}(\mathbb{R}_+\times \mathbb{R}_+\times S_h,\mathbb{R})$ with $h\geq 2c$ and $\,\mu \in \mathcal{K}$ such that for all $(t_1,t_2,x) \in \mathbb{R}_+\times \mathbb{R}_+\times S_h$
\begin{equation}
\begin{aligned}
&1.\ V(t_1,t_2,0)=0,\\
&2.\ \mu(|x|)\leq V(t_1,t_2,x),\\
&3.\ L_1V(t_1,t_2,x) \leq -\gamma_1(\alpha)\ a.s. \ and\ L_2V(t_1,t_2,x) \leq -\gamma_2(\alpha)\ a.s., \ for \ any \ \alpha\in (0,h),\\
&\ where\ \gamma_1(\alpha) \geq 0\ and\ \gamma_2(\alpha) \geq 0\ but\ not\ equal\ to\ zero\ at\ the\ same\ time, \ x\in S_h- \bar{S_\alpha},
\end{aligned}
\end{equation}
then the trivial solution of the time-changed SDE \eqref{SDE} is stochastically  asymptotically stable.
\end{tm}

\begin{proof}
See Section \ref{appd}.
\end{proof}

\begin{tm}\label{tm3}
Assume that there exists a function $V(t_1,t_2,x)\in C^{1,1,2}(\mathbb{R}_+\times \mathbb{R}_+\times \mathbb{R},\mathbb{R})$ and $\,u\in \mathcal{K}$ such that for all $(t_1,t_2,x) \in \mathbb{R}_+\times \mathbb{R}_+\times \mathbb{R}$
\begin{equation}
\begin{aligned}
&1.\ V(t_1,t_2,0)=0,\\
&2.\ \mu(|x|)\leq V(t_1,t_2,x),\\
&3.\ L_1V(t_1,t_2,x) \leq -\gamma_1(x)\ a.s. \ and\ L_2V(t_1,t_2,x) \leq -\gamma_2(x)\ a.s., \\
&\ where\ \gamma_1(x) \geq 0\ and\ \gamma_2(x) \geq 0\ but\ not\ equal\ to\ zero\ at\ the\ same\ time,\\
&4.\ \lim_{|x|\rightarrow \infty}\inf_{t_1,t_2\geq 0}V(t_1,t_2,x)=\infty,
\end{aligned}
\end{equation}
then the trivial solution of the time-changed SDE \eqref{SDE} is globally stochastically asymptotically stable.
\end{tm}

\begin{proof}
This proof has similar idea as Theorem 4.2.4 in \cite{maotext}, so we omit the details here.
\end{proof}

\begin{exmp}
Consider the following SDE driven by time-changed L\'evy noise\\
\begin{equation}\label{example1}
\begin{aligned}
dX(t)=&f(t,E_t)X(t)dt+k(t,E_t)X(t)dE_t\\
&+g(t,E_t)X(t)dB_{E_t}+\int_{|y|<c}h(t,E_t,y)X(t)d\tilde{N}(dE_s,dy)
\end{aligned}
\end{equation}
with $X(0)=x_0$, where $k,f,g,h$ are $\mathcal{G}_t$-measurable real-valued functions satisfying Lipschitz condition \ref{lip} and assumption \ref{tec}. Define Lyapunov function
\begin{equation}
V(t_1,t_2,x)=|x|^{\alpha}
\end{equation}
on $\mathbb{R}_+\times \mathbb{R}_+\times \mathbb{R}$ for some $\alpha\in (0,1)$. Then
\begin{equation}
L_1V(t_1,t_2,x)=\alpha f(t_1,t_2) |x|^\alpha
\end{equation}
and
\begin{equation}
\begin{aligned}
L_2V(t_1,t_2,x)=\bigg[\alpha k(t_1,t_2)&+\frac{\alpha(\alpha-1)}{2}g^2(t_1,t_2)\\
&+\int_{|y|<c}\Big[|1+h(t_1,t_2,y)|^\alpha-1-\alpha h(t_1,t_2,y)\Big]\nu(dy)\bigg]|x|^\alpha.
\end{aligned}
\end{equation}
Thus, if
\begin{equation}
\alpha f(t,E_t) \leq 0 \ \ a.s.
\end{equation}
and
\begin{equation}
\alpha k(t,E_t) +\frac{\alpha(\alpha-1)}{2}g^2(t,E_t)+\int_{|y|<c}\Big[|1+h(t,E_t,y)|^\alpha-1-\alpha h(t,E_t,y)\Big]\nu(dy) \leq 0 \ \ a.s.
\end{equation}
for all $t,E_t\in \mathbb{R}_+$, the trivial solution of SDE \eqref{example1} is stochastically stable, by Theorem \ref{tm1}.

Let $\alpha=0.5,\ c=1$ and $f(t_1,t_2)=-1,\ k(t_1,t_2)=0.25,\ g(t_1,t_2)=1,\ h(t_1,t_2,y)=y$ for all $t_1,t_2 \in \mathbb{R}_+$, then
 \begin{equation}
L_1V(t_1,t_2,x)=-\frac{|x|^\alpha}{2}\leq 0
\end{equation}
and
\begin{equation}
L_2V(t_1,t_2,x)=\int_{|y|<1}\Big[|1+y|^\frac{1}{2}-1-\frac{1}{2}y\Big]\nu(dy)<0.
\end{equation}
Therefore, by Theorem \ref{tm3}, trivial solution of SDE
\begin{equation}
dX(t)=-X(t)dt+0.25X(t)dE_t+X(t)dB_{E_t}+\int_{|y|<1}yX(t)d\tilde{N}(ds,dy)
\end{equation}
with $X(0)=x_0$ is globally stochastically asymptotically stable.

\end{exmp}

\begin{tm}\label{tm4}
Let $p,\alpha_1,\alpha_2,\alpha_3$ be positive constants. If $V\in C^2(\mathbb{R}_+\times \mathbb{R}_+\times \mathbb{R}; \mathbb{R}_+)$ satisfies
\begin{equation}
\begin{aligned}
&1.\ V(t_1,t_2,0)=0,\ \ \ \ \ 2.\ \alpha_1|x|^p \leq V(t_1,t_2,x) \leq \alpha_2|x|^p, \\
&3.\ L_2V(t_1,t_2,x) \leq 0,\ \ 4.\ L_1V(t_1,t_2,x) \leq -\alpha_3V(t_1,t_2,x),
\end{aligned}
\end{equation}
$\forall (t_1,t_2,x) \in \mathbb{R}_+\times \mathbb{R}_+\times \mathbb{R},$ then the trivial solution of the time-changed SDE \eqref{SDE} is pth moment exponentially stable with
\begin{equation}
\mathbb{E}|X(t,x_0)|^p\leq \frac{\alpha_2}{\alpha_1}|x_0|^p \exp(-\alpha_3 t).
\end{equation}
\end{tm}

\begin{proof}
See Section \ref{appd}.
\end{proof}

\begin{exmp}
Consider the following SDE driven by time-changed L\'evy noise
\begin{equation}
dX(t)=-X(t)dt+E_tdB_{E_t}+\int_{|y|<1}\Big[X(t)y^2-X(t)\Big]\tilde{N}(dE_t,dy)
\end{equation}
with $X(0)=x_0$ and $\nu$ is a L\'evy measure.
Let $V(t_1,t_2,x)=|x|$, then
\begin{equation}
L_1V(t_1,t_2,x)=-|x|
\end{equation}
and
\begin{equation}
\begin{aligned}
L_2V(t_1,t_2,x)&=\int_{|y|<1}\Big[|x+xy^2-x|-|x|- sgn(x)(xy^2-x)\Big]\nu(dy)\\
&=\int_{|y|<1}\Big[(|y^2|-y^2)|x|\Big]\nu(dy)=0.
\end{aligned}
\end{equation}

By Theorem \ref{tm4}, X(t) is first moment exponentially stable, that is,
\begin{equation}
\mathbb{E}|X(t,x_0)|\leq |x_0|\exp(-t), \forall t\geq 0.
\end{equation}
\end{exmp}

Next, we reduce SDE \eqref{SDE} by setting $f(t,E_t,X(t-))=0$,
\begin{equation}\label{redSDE}
dX(t)=k(E_t, X(t-))dE_t+g(E_t, X(t-))dB_{E_t}+\int_{|y|<c}h(E_t, X(t-),y)\tilde{N}(dE_t,dy),\\
\end{equation}
with  $X(0)=x_0$.

Kobayashi \cite{keib} mentioned duality related to \eqref{redSDE} and the following SDE
\begin{equation}\label{redSDEsim}
dY(t)=k(t, Y(t-))dt+g(t, Y(t-))dB_{t}+\int_{|y|<c}h(t, Y(t-),y)\tilde{N}(dt,dy),\\
 Y(0)=x_0,
\end{equation}
with $Y(0)=x_0$,
stating that\\
1. If a process $Y(t)$ satisfies SDE \eqref{redSDEsim}, then $X(t):=Y(E_t)$ satisfies the time-changed SDE \eqref{redSDE};\\
2. If a process $X(t)$ satisfies the time-changed SDE \eqref{redSDE}, then $Y(t):=X(D(t))$ satisfies SDE \eqref{redSDEsim}.

\begin{coro}
Let $Y(t)$ be a stochastically stable (stochastically asymptotically stable, globally stochastically asymptotically stable) process satisfying SDE \eqref{redSDEsim}, then the trivial solution $X(t)$ of SDE \eqref{redSDE} is a stochastically stable (stochastically asymptotically stable, globally stochastically asymptotically stable) process, respectively.
\end{coro}

\begin{proof}
This proof has similar idea as Corollary 3.1 in\cite{qwa}, thus we omit details.

\begin{rm}
Though the conclusion of Corollary 3.1 in\cite{qwa} is correct, there is a minor problem in the proof. We correct it as following
\begin{equation}
\begin{aligned}
P\Big\{|X(t,x_0)|<h, \forall t\geq 0\Big\}&=P\Big\{|Y(E_t,x_0)|<h, \forall t\geq 0\Big\}\\
&=P\Big\{\sup_{0\leq t< \infty}|Y(E_t,x_0)|<h\Big\}\\
&=P\Big\{\sup_{ \{E_t :\ 0\leq t<\infty\} }|Y(E_t,x_0)|<h\Big\}\\
&=P\Big\{\sup_{0\leq \tau< \infty}|Y(\tau,x_0)|<h\Big\}\\
&=P\Big\{|Y(t,x_0)|<h, \forall t\geq 0\Big\}\\
&=1-\epsilon.
\end{aligned}
\end{equation}
\end{rm}
Here, we use the fact that the image of $[0,\infty)$ under $E_t$ process is almost surely equal to $[0,\infty)$.

\end{proof}

\begin{coro}\label{tm5}
Let $Y(t)$ be a pth moment exponentially stable process satisfying SDE \eqref{redSDEsim}, the $X(t)$ is a pth moment asymptotically stable satisfying SDE \eqref{redSDE}.
\end{coro}

\begin{proof}
See Section \ref{appd}.
\end{proof}

\begin{rk}
Existence of pth moment stability of the solution of SDE \eqref{redSDEsim} has been proved by Theorem 3.5.1 in Siakalli \cite{sia}.
\end{rk}


\section{Proofs of main results} \label{appd}
\subsection {Proof of Theorem \ref{tm1}}
\begin{proof}\label{proofoftm1}
Let $\epsilon \in (0,1)$ and $r \in (0, h)$ be arbitrary. By continuity of $V(t_1,t_2,x)$ and the fact $V(t_1,t_2,0)=0$, we can find a $\delta=\delta(\epsilon, r, 0)>0$ such that
\begin{equation}\label{equini}
\frac{1}{\epsilon}\sup_{x\in S_\delta}V(0,0,x_0)\leq \mu(r).
\end{equation}

By \eqref{equini} and condition (2), $\delta<r$. Fix initial value $x_0\in S_\delta$ arbitrarily and define the stopping time
\begin{equation}\label{stoptime}
\tau_r=\inf\{t\geq 0: |X(t,x_0)| \geq r\},
\end{equation}
where $r\leq \frac{h}{2}$, and
\begin{equation}\label{stoptime2}
\begin{aligned}
U_k=&k \land \inf\{ t\geq 0;\Bigg|\int_0^{\tau_r \land t}V_x(s,E_s,X(s-))g(s,E_s,X(s-))dB_{E_s}\Bigg|\geq k\},\\
W_k=&k \land \inf\{ t\geq 0; \Bigg|\int_0^{\tau _r\land t}\int_{|y|<c}\bigg[ V(s,E_s,X(s-)+H(s,E_s,X(s-),y))\\
&\hspace{7cm}-V(s,E_s,X(s-))\bigg] \tilde{N}(dE_s,dy)\Bigg|\geq k\},
\end{aligned}
\end{equation}
for k=1,2,.... It is easy to see that $U_k\rightarrow \infty$ and $W_k\rightarrow \infty$ as $k\rightarrow \infty$. Apply It\^o formula \eqref{itolevy} to $V(t_1,t_2,x)$ associated with SDE \eqref{SDE}, then for any $t\geq 0$,
\begin{equation}
\begin{aligned}
&V(t \land \tau_r  \land U_k \land W_k,E_{t \land \tau_r\land U_k \land W_k},X(t \land \tau_r\land U_k \land W_k))-V(0,0,x_0)\\
=&\int_0^{t \land \tau_r\land U_k \land W_k}L_1V(s,E_s,X(s-))ds\\
+&\int_0^{t \land \tau_r\land U_k \land W_k}L_2V(s,E_s,X(s-))dE_s+\int_0^{t \land \tau_r\land U_k \land W_k}V_x(s,E_s,X(s-))g(s,E_s,X(s-))dB_{E_s}\\
+&\int_0^{t \land \tau_r\land U_k \land W_k}\int_{|y|<c}\bigg[V(s,E_s,X(s-)+H(s,E_s,X(s-),y))-V(s,E_s,X(s-))\bigg]\tilde{N}(dE_s,dy).
\end{aligned}
\end{equation}

By \cite{magpath} and \cite{hhk}, both
\begin{equation}
\int_0^{t \land \tau_r\land U_k \land W_k}V_x(s,E_s,X(s-))g(s,E_s,X(s-))dB_{E_s}
\end{equation}
and
\begin{equation}
\int_0^{t \land \tau_r\land U_k \land W_k}\int_{|y|<c}\bigg[V(s,E_s,X(s-)+H(s,E_s,X(s-),y))-V(s,E_s,X(s-))\bigg]\tilde{N}(dE_s,dy)
\end{equation}
are mean zero martingales.

Taking expectations on both sides, we have
$$\mathbb{E}[V(t \land \tau_r\land U_k \land W_k,E_{t \land \tau_r\land U_k \land W_k},X(t \land \tau_r\land U_k \land W_k))]\leq V(0,0,x_0).$$

Letting $k\rightarrow \infty$,
$$\mathbb{E}[V(t \land \tau_r,E_{t \land \tau_r},X(t \land \tau_r))]\leq V(0,0,x_0).$$

Now, $|X(t \land \tau_r)|<r$ for $t<\tau_r$. For all $w\in\{\tau_r<\infty\}$, $|X(\tau_r)(w)|\leq r+c\leq h$.
Since $V(t_1,t_2,x)\geq \mu(|x|)$ for all $x\in S_h$, we have for all $w\in\{\tau_r<\infty\}$
\begin{equation}
V(\tau_r,E_{\tau_r},X(\tau_r)(w))\geq \mu(|X(\tau_r)(w)|)\geq \mu(r).
\end{equation}

Also,
\begin{equation}
V(0,0,x_0)\geq E[V(t \land \tau_r,E_{t \land \tau_r},X(t \land \tau_r))1_{\{\tau_r<t\}}]\geq E[\mu(r)1_{\{\tau_r<t\}}]=\mu(r)P(\tau_r<t),
\end{equation}

thus, combined with \eqref{equini},
\begin{equation}
P(\tau_r<t)\leq \frac{V(0,0,x_0)}{\mu(r)}\leq \frac{\epsilon \mu(r)}{\mu(r)}=\epsilon.
\end{equation}

Then, letting $t\rightarrow \infty$, we have
\begin{equation}
P(\tau_r<\infty)\leq \epsilon,
\end{equation}

equivalently,
\begin{equation}
P(|X(t,x_0)|<r \ for \ all \ t\geq 0)\geq  1-\epsilon,
\end{equation}

so $X(t,x_0)$ is stochastically stable.
\end{proof}


\subsection {Proof of Theorem \eqref{tm2}}
\begin{proof}
By Theorem \ref{tm1}, trivial solution of \eqref{SDE} is stochastically stable. For any fixed $\epsilon \in (0,1)$, there exists $\delta=\delta(\epsilon)>0$ such that
\begin{equation}
P(|X(t,x_0)|<h)\geq 1-\frac{\epsilon}{5}
\end{equation}
when $x_0\in S_\delta$. Fix $x_0\in S_\delta$ and let $0<\alpha<\beta<|x_0|$ arbitrarily. Define the following stopping times
\begin{equation}
\begin{aligned}
&\tau_h=\inf\{t\geq 0;|X(t,x_0)|>h\}\\
&\tau_\alpha=\inf\{t\geq 0;|X(t,x_0)|<\alpha\}\\
&U_k=k\land \inf\{t\geq 0;\Bigg| \int_0^{t \land \tau_h \land \tau_\alpha}V_x(s,E_s,X(s-))g(s,E_s,X(s-))dB_{E_s}\Bigg|\geq k\},\\
&W_k=k\land \inf\{t\geq 0;\Bigg| \int_0^{t \land \tau_h \land \tau_\alpha}\int_{|y|<c}\big[V_x(s,E_s,X(s-)+h(s,E_s,X(s-),y))\\
&\hspace{8.5cm}-V_x(s,E_s,X(s-))\big]\tilde{N}(dE_s,dy)\Bigg|\geq k\}.
\end{aligned}
\end{equation}

By It\^o's formula \eqref{itolevy}, we have
\begin{equation}
\begin{aligned}
0&\leq \mathbb{E}\big[ V(t \land \tau_h \land \tau_\alpha \land U_k \land W_k,E_{t \land \tau_h \land \tau_\alpha \land U_k \land W_k}, X(t \land \tau_h \land \tau_\alpha \land U_k \land W_k)) \big]\\
&=V(0,0,x_0)+ \mathbb{E}\int_0^{t \land \tau_h \land \tau_\alpha \land U_k \land W_k}L_1V(s,E_s,X(s-))ds\\
&\hspace{7.5cm}+\mathbb{E}\int_0^{t \land \tau_h \land \tau_\alpha \land U_k \land W_k}L_2V(s,E_s,X(s-))dE_s\\
&\leq V(0,0,x_0)-\gamma_1(\alpha)\mathbb{E}[t \land \tau_h \land \tau_\alpha \land U_k \land W_k]-\gamma_2(\alpha)\mathbb{E}[E_{t \land \tau_h \land \tau_\alpha \land U_k \land W_k}].
\end{aligned}
\end{equation}

Letting $k\rightarrow \infty$ and $t\to\infty$, we have
\begin{equation}
\gamma_1(\alpha)\mathbb{E}[ \tau_h \land \tau_\alpha]+\gamma_2(\alpha)\mathbb{E}[E_{  \tau_h \land \tau_\alpha}]\leq V(0,0,x_0),
\end{equation}



By condition (3) and $E_t\rightarrow \infty$ a.s. as $t\rightarrow \infty$, see proof of Theorem \ref{tm4},
we have
\begin{equation}
P(\tau_h \land \tau_\alpha< \infty)=1.
\end{equation}

Since $P(\tau_h=\infty)>1-\frac{\epsilon}{5}$, it follows that $P(\tau_h<\infty)\leq \frac{\epsilon}{5}$, thus
\begin{equation}
1=P(\tau_h \land \tau_\alpha< \infty)\leq P(\tau_h<\infty)+P(\tau_\alpha<\infty)\leq P(\tau_\alpha<\infty)+\frac{\epsilon}{5},
\end{equation}
that's,
\begin{equation}
P(\tau_\alpha<\infty)\geq 1-\frac{\epsilon}{5}.
\end{equation}

Choose $\theta$ sufficiently large for
\begin{equation}
P(\tau_\alpha<\theta)\geq 1-\frac{2\epsilon}{5}.
\end{equation}
Then
\begin{equation}
\begin{aligned}
P(\tau_\alpha<\tau_h\land \theta)&\geq P(\{\tau_\alpha<\theta\}\cap\{\tau_h=\infty\})=P(\tau_\alpha<\theta)-P(\{\tau_\alpha<\theta\}\cap\{\tau_h<\infty\})\\
&\geq P(\tau_\alpha<\theta)-P(\tau_h<\infty)\geq 1-\frac{2\epsilon}{5}-\frac{\epsilon}{5}=1-\frac{3\epsilon}{5}
\end{aligned}
\end{equation}

Now define some stopping times
\begin{eqnarray}\sigma=
\begin{cases}
\tau_\alpha, & if\ \tau_\alpha<\tau_h \land \theta \cr
\infty, & otherwise
\end{cases}
\end{eqnarray}

\begin{equation}
\begin{aligned}
&\tau_\beta=\inf\{t\geq \sigma;|X(t,x_0)|\geq \beta\},\\
&S_i=\inf\{t\geq \sigma;\big|\int_\sigma^{\tau_\beta\land t}V_x(s,E_s,X(s-))g(s,E_s,X(s-))dB_{E_s} \big|\geq i \},\\
&T_i=\inf\{t\geq \sigma;\big|\int_\sigma^{\tau_\beta\land t}\int_{|y|<c}\big[V(s,E_s,X(s-)+h(s,E_s,X(s-),y))\\
&\hspace{7cm}-V(s,E_s,X(s-))\big]\tilde{N}(dE_s,dy) \big|\geq i \}.\\
\end{aligned}
\end{equation}

Again, by It\^o's formula,
\begin{equation}
\begin{aligned}
&\mathbb{E}\bigg[V(t \land \tau_\beta  \land S_i \land T_i, E_{t \land \tau_\beta \land S_i \land T_i},X(t \land \tau_\beta  \land S_i \land T_i))\bigg]\\
\leq &\mathbb{E}\bigg[V(t \land \sigma ,E_{t \land \sigma },X(t \land \sigma  \land))\bigg]+\mathbb{E}\bigg[ \int_{t \land \sigma \land }^{t \land \tau_\beta \land S_i \land T_i}L_1V(s,E_s,X(s-))ds\bigg]\\
&+\mathbb{E}\bigg[ \int_{t \land \sigma \land }^{t \land \tau_\beta \land S_i \land T_i}L_2V(s,E_s,X(s-))dE_s\bigg]\\
\leq&\mathbb{E}\bigg[V(t \land \sigma,E_{t \land \sigma},X(t \land \sigma))\bigg].
\end{aligned}
\end{equation}

Letting $i\rightarrow \infty$,
\begin{equation}
\mathbb{E}\bigg[V(\sigma \land t,E_{\sigma \land t},X(\sigma \land t))\bigg]\geq \mathbb{E}\bigg[V(\tau_\beta \land t,E_{\tau_\beta \land t},X(\tau_\beta \land t))\bigg],
\end{equation}

that is,
\begin{equation}
\mathbb{E}\bigg[V(\sigma \land t,E_{\sigma \land t},X(\sigma \land t))[\mathbbm{1}_{\{\sigma<\infty\}}+\mathbbm{1}_{\{\sigma=\infty\}}]\bigg]\geq \mathbb{E}\bigg[V(\tau_\beta \land t,E_{\tau_\beta \land t},X(\tau_\beta \land t))[\mathbbm{1}_{\{\sigma<\infty\}}+\mathbbm{1}_{\{\sigma=\infty\}}]\bigg].
\end{equation}
For $w\in \{\tau_\alpha\geq \tau_h \land \theta\}$, we have $\sigma=\infty$, then $\tau_\beta=\infty$, thus
\begin{equation}
V(\sigma \land t,E_{\sigma \land t},X(\sigma \land t))=V(t,E_{t}, X(t))
\end{equation}
and
\begin{equation}
V(\tau_\beta \land t,E_{\tau_\beta \land t},X(\tau_\beta \land t))=V(t,E_{t}, X(t))
\end{equation}

Thus,
\begin{equation}\label{a}
\mathbb{E}\bigg[V(\sigma \land t,E_{\sigma \land t},X(\sigma \land t))\mathbbm{1}_{\{\sigma<\infty\}}\bigg]\geq \mathbb{E}\bigg[V(\tau_\beta \land t,E_{\tau_\beta \land t},X(\tau_\beta \land t))\mathbbm{1}_{\{\sigma<\infty\}}\bigg].
\end{equation}

Now, focus on the right hand side of \eqref{a}, by definition of $\tau_\beta$, $\tau_\beta\geq \sigma$, thus \\ $\mathbbm{1}_{\{\sigma<\infty\}}\geq \mathbbm{1}_{\{\tau_\beta<\infty\}}$, then
\begin{equation}\label{b}
\mathbb{E}\bigg[V(\tau_\beta \land t,E_{\tau_\beta \land t},X(\tau_\beta \land t))\mathbbm{1}_{\{\sigma<\infty\}}]\bigg]\geq \mathbb{E}\bigg[V(\tau_\beta \land t,E_{\tau_\beta \land t},X(\tau_\beta \land t))\mathbbm{1}_{\{\tau_\beta<\infty\}}]\bigg].
\end{equation}

Combining \eqref{a} and \eqref{b}, we have
\begin{equation}
\mathbb{E}\bigg[V(\sigma \land t,E_{\sigma \land t},X(\sigma \land t))\mathbbm{1}_{\{\sigma<\infty\}}\bigg]\geq \mathbb{E}\bigg[V(\tau_\beta \land t,E_{\tau_\beta \land t},X(\tau_\beta \land t))\mathbbm{1}_{\{\tau_\beta<\infty\}}]\bigg].
\end{equation}

Since $P(\sigma<\infty)=P(\tau_\alpha<\tau_h\land \theta)$ and $P(\tau_\beta<\infty)\geq P(\{\tau_\beta<\infty\}\cap\{\tau_h=\infty\})$, it follows that
\begin{equation}\label{c}
\mathbb{E}\big[V(\tau_\beta,E_{\tau_\beta},X(\tau_\beta))\mathbbm{1}_{\{\tau_\beta<\infty\}\cap\{\tau_h=\infty\}}\big] \leq \mathbb{E}\big[ V(\tau_\alpha,E_{\tau_\alpha},X(\tau_\alpha)) \mathbbm{1}_{\{\tau_\alpha<\tau_h\land \theta\}} \big].
\end{equation}

By condition (2)
\begin{equation}
0\leq\mu(|x|)\leq V(t_1,t_2,x),
\end{equation}
for all $(t_1,t_2,x) \in \mathbb{R}_+\times \mathbb{R}_+\times \mathbb{R}$, and
$|X(\tau_\beta)|\geq \beta>0$.

Then, for the left hand side of \eqref{c}, we have
\begin{equation}\label{d}
\begin{aligned}
\mathbb{E}\big[V(\tau_\beta,E_{\tau_\beta},X(\tau_\beta))\mathbbm{1}_{\{\tau_\beta<\infty\}\cap\{\tau_h=\infty\}}\big]&\geq \mathbb{E}\big[\mu(|X(\tau_\beta)|)\mathbbm{1}_{\{\tau_\beta<\infty\}\cap\{\tau_h=\infty\}}\big]\\
&\geq\mathbb{E}\big[\mu(\beta)\mathbbm{1}_{\{\tau_\beta<\infty\}\cap\{\tau_h=\infty\}}\big]\\
&=\mu(\beta)\mathbb{E}\big[\mathbbm{1}_{\{\tau_\beta<\infty\}\cap\{\tau_h=\infty\}}\big]\\
&=\mu(\beta)P(\{\tau_\beta<\infty\}\cap\{\tau_h=\infty\}).
\end{aligned}
\end{equation}

Let
\begin{equation}
B_\alpha=\sup_{t_1\times t_2 \times x\in \mathbb{R}_+ \times \mathbb{R}_+ \times \bar{S_\alpha}} V(t_1,t_2,x),
\end{equation}
then $B_\alpha\rightarrow 0$ as $\alpha\rightarrow 0$, that's, $\frac{B_\alpha}{\mu(\beta)}< \frac{\epsilon}{5}$ for some $\alpha$.

For the right hand side of \eqref{c},
\begin{equation}\label{e}
\begin{aligned}
\mathbb{E}\big[ V(\tau_\alpha,E_{\tau_\alpha},X(\tau_\alpha)) \mathbbm{1}_{\{\tau_\alpha<\tau_h\land \theta\}} \big]&\leq \mathbb{E}\big[ B_\alpha \mathbbm{1}_{\{\tau_\alpha<\tau_h\land \theta\}} \big]\\
&=B_\alpha  \mathbb{E}\big[ \mathbbm{1}_{\{\tau_\alpha<\tau_h\land \theta\}} \big]\\
&=B_\alpha P(\tau_\alpha<\tau_h\land \theta).
\end{aligned}
\end{equation}

Combining \eqref{d} and \eqref{e}, we have
\begin{equation}
P(\{\tau_\beta<\infty\}\cap\{\tau_h=\infty\})\mu(\beta)\leq B_\alpha P(\tau_\alpha<\tau_h\land \theta),
\end{equation}

thus
\begin{equation}
P(\{\tau_\beta<\infty\}\cap\{\tau_h=\infty\})\leq \frac{ B_\alpha}{\mu(\beta)}P(\tau_\alpha<\tau_h\land \theta)<\frac{\epsilon}{5}.
\end{equation}

Also,
\begin{equation}
P(\{\tau_\beta<\infty\}\cap\{\tau_h=\infty\})\geq P(\tau_\beta<\infty)-P(\tau_h<\infty)>P(\tau_\beta<\infty)-\frac{\epsilon}{5},
\end{equation}
so,
\begin{equation}
P(\tau_\beta<\infty)<\frac{2\epsilon}{5}.
\end{equation}

Next \begin{equation}
\begin{aligned}
P(\{\sigma<\infty\}\cap\{\tau_\beta=\infty\})&\geq P(\sigma<\infty)-P(\tau_\beta<\infty)\\
&>P(\tau_\alpha<\tau_h\land \theta)-\frac{2\epsilon}{5}\\
&\geq 1- \frac{3\epsilon}{5}-\frac{2\epsilon}{5}\\
&=1- \epsilon.
\end{aligned}
\end{equation}

Hence,

\begin{equation}
P\{\omega;\limsup_{t\rightarrow \infty}\big|X(t,x_0)\big|\leq \beta\}>1-\epsilon.
\end{equation}

Since $\beta$ is arbitrary, we have
\begin{equation}
P\{\omega;\limsup_{t\rightarrow \infty}\big|X(t,x_0)\big|=0\}>1-\epsilon,
\end{equation}

as desired.

\end{proof}


\subsection {Proof of Theorem \eqref{tm4}}
\begin{proof}\label{proofoftm4}
Define a function $Z: \mathbb{R}_+\times \mathbb{R}_+\times \mathbb{R}\rightarrow \mathbb{R}_+$ by
\begin{equation}
Z(t_1,t_2,x)=exp(\alpha_3t_1)V(t_1,t_2,x).
\end{equation}
Fix any $x_0\neq 0$ in $\mathbb{R}$. For each $n \geq |x_0|$, define
$$\tau_n=\inf\{t\geq 0:|X(t)|\geq n\},$$

and
\begin{equation}\label{stoptime2}
\begin{aligned}
U_k=&k \land \inf\{ t\geq 0;\Bigg|\int_0^{\tau_n \land t}V_x(s,E_s,X(s-))g(s,E_s,X(s-))dB_{E_s}\Bigg| \geq k\},\\
W_k=&k \land \inf\{ t\geq 0; \Bigg|\int_0^{\tau _n\land t}\int_{|y|<c}\bigg[ V(s,E_s,X(s-)+h(s,E_s,X(s-),y))\\
&\hspace{7cm}-V(s,E_s,X(s-))\bigg] \tilde{N}(ds,dy)\Bigg|\geq k\},
\end{aligned}
\end{equation}
for k=1,2,.... It is easy to see that $U_k\rightarrow \infty$ and $W_k\rightarrow \infty$ as $k\rightarrow \infty$.

Apply It\^o formula \eqref{itolevy} to $Z(\tau_n\land U_k\land W_k,E_{\tau_n\land U_k\land W_k},X(\tau_n\land U_k\land W_k))$,  then we have
\begin{equation}
\begin{aligned}
&Z(t \land \tau_n\land U_k\land W_k,E_{t \land \tau_n\land U_k\land W_k},X(t \land \tau_n\land U_k\land W_k))-Z(0,0,x_0)\\
=&\int_0^{t \land \tau_n\land U_k\land W_k}\exp(\alpha_3s)\bigg[\alpha_3 V(s,E_s,X(s-)) + V_s(s,E_s,X(s-)\bigg])ds\\
&+\int_0^{t \land \tau_n\land U_k\land W_k}\exp(\alpha_3s)V_{E_s}(s,E_s,X(s-))dE_s\\
&+\int_0^{t \land \tau_n\land U_k\land W_k}\exp(\alpha_3s)V_x(s,E_s,X(s-))\bigg[f(s,E_s,X(s-))dt\\
&\hspace{7cm}+k(s,E_s,X(s-))dE_t+g(s,E_s,X(s-))dB_{E_t}\bigg]\\
&+\frac{1}{2}\int_0^{t \land \tau_n\land U_k\land W_k}\exp(\alpha_3s)V_{xx}(s,E_s,X(s-))g^2(s,E_s,X(s-))dE_s\\
&+\int_0^{t \land \tau_n\land U_k\land W_k}\int_{|y|<c}\exp(\alpha_3s)\bigg[V(s,E_s,X(s-)+h(s,E_s,X(s-),y))\\
&\hspace{9.5cm}-V(s,E_s,X(s-))\bigg]\tilde{N}(dE_s,dy)\\
&+\int_0^{t \land \tau_n\land U_k\land W_k}\int_{|y|<c}\exp(\alpha_3s)\bigg[V(s,E_s,X(s-)+h(s,E_s,X(s-),y))-V(s,E_s,X(s-))\\
&\hspace{6.5cm}-V_x(s,E_s,X(s-))h(s,E_s,X(s-),y)\bigg]\nu(dy)dE_s\\
\end{aligned}
\end{equation}

\begin{equation*}
\begin{aligned}
=&\int_0^{t \land \tau_n\land U_k\land W_k}\exp(\alpha_3s)\bigg[\alpha_3 V(s,E_s,X(s-)) + V_s(s,E_s,X(s-))\\
&\hspace{8cm}+V_x(s,E_s,X(s-))f(s,E_s,X(s-))\bigg])ds\\
&+\int_0^{t \land \tau_n\land U_k\land W_k}\exp(\alpha_3s)\Bigg[V_{E_s}(s,E_s,X(s-))+V_x(s,E_s,X(s-))k(s,E_s,X(s-))\\
&+\frac{1}{2}V_{xx}(s,E_s,X(s-))g^2(s,E_s,X(s-))+\int_{|y|<c}\bigg[V(s,E_s,X(s-)+h(s,E_s,X(s-),y))\\
&\hspace{3cm}-V(s,E_s,X(s-))-V_x(s,E_s,X(s-))h(s,E_s,X(s-),y)\bigg]\nu(dy)\Bigg]dE_s\\
&+\int_0^{t \land \tau_n\land U_k\land W_k}\exp(\alpha_3s)g(s,E_s,X(s-))V_x(s,E_s,X(s-))dB_{E_s}\\
&+\int_0^{t \land \tau_n\land U_k\land W_k}\int_{|y|<c}\exp(\alpha_3s)\bigg[V(s,E_s,X(s-)+h(s,E_s,X(s-),y))\\
&\hspace{9.5cm}-V(s,E_s,X(s-))\bigg]\tilde{N}(dE_s,dy)\\
\end{aligned}
\end{equation*}
\begin{equation*}
\begin{aligned}
=&\int_0^{t \land \tau_n\land U_k\land W_k}\exp(\alpha_3s)\bigg[\alpha_3 V(s,E_s,X(s-))+L_1V(s,E_s,X(s-)\bigg]ds\\
&+\int_0^{t \land \tau_n\land U_k\land W_k}\exp(\alpha_3s)L_2V(s,E_s,X(s-))dE_s\\
&+\int_0^{t \land \tau_n\land U_k\land W_k}\exp(\alpha_3s)g(s,E_s,X(s-))V_x(s,E_s,X(s-))dB_{E_s}\\
&+\int_0^{t \land \tau_n\land U_k\land W_k}\int_{|y|<c}\exp(\alpha_3s)\bigg[V(s,E_s,X(s-)+h(s,E_s,X(s-),y))\\
&\hspace{9.5cm}-V(s,E_s,X(s-))\bigg]\tilde{N}(dE_s,dy)\\
\end{aligned}
\end{equation*}

By similar ideas as in the proof of \eqref{proofoftm1}, we have that
$$\int_0^{t \land \tau_n\land U_k\land W_k}\exp(\alpha_3s)g(s,E_s,X(s-))V_x(s,E_s,X(s-))dB_{E_s}$$
and
$$\int_0^{t \land \tau_n\land U_k\land W_k}\int_{|y|<c}\exp(\alpha_3s)\bigg[V(s,E_s,X(s-)+h(s,E_s,X(s-),y))-V(s,E_s,X(s-))\bigg]\tilde{N}(dE_s,dy)$$
are mean zero martingales.
Taking expectations on both sides, we have
\begin{equation}
\begin{aligned}
&\mathbb{E}[\exp(\alpha_3(t \land \tau_n\land U_k\land W_k))V(t \land \tau_n\land U_k\land W_k,E_{t \land \tau_n\land U_k\land W_k},X(t \land \tau_n\land U_k\land W_k))]\\
\leq &\mathbb{E}\int_0^{t \land \tau_n\land U_k\land W_k}\exp(\alpha_3s)\bigg[\alpha_3 V(s,E_s,X(s-))+L_1V(s,E_s,X(s-)\bigg]ds+V(0,0,x_0) \\
\leq & V(0,0,x_0).
\end{aligned}
\end{equation}

Letting $k\rightarrow \infty$ and  $n\rightarrow \infty$, $\mathbb{E}[\exp(\alpha_3t)V(t,E_t,X(t))]\leq V(0,0,x_0)$.
By condition (2),
\begin{equation}
\alpha_1 |X(t)|^p\leq V(t,E_t,X(t)),
\end{equation}
then
\begin{equation}
\alpha_1 \mathbb{E}(\exp(\alpha_3t)|X(t)|^p)\leq \mathbb{E}(\exp(\alpha_3s)V(t,E_t,X(t)))\leq V(0,0,x_0)\leq \alpha_2|x_0|^p,
\end{equation}
that's
\begin{equation}
\mathbb{E}(|X(t)|^p)\leq \frac{\alpha_2}{\alpha_1}\exp(-\alpha_3t)|x_0|^p,
\end{equation}
as desired.
\end{proof}

\subsection{Proof of Corollary \ref{tm5}}
\begin{proof}
If $Y(t)$ satisfies SDE \eqref{redSDEsim}, by Theorem 4.2 in \cite{keib}, $X(t)=Y(E_t)$ satisfies \eqref{redSDE}.\\
Since $Y(t)$ is pth moment exponentially stable, there exist two positive constants $\lambda$ and C such that
\begin{equation}
\mathbb{E}[|X(t)|^p]\leq C|x_0|^p\exp(-\lambda t),\ \forall t\geq 0,\ \forall x_0\in \mathbb{R},\ p>0,
\end{equation}
then
\begin{equation}
\begin{aligned}
\mathbb{E}[|Y(t)|^p]&=\mathbb{E}[|X(E_t)|^p]\\
&=\int_0^\infty \mathbb{E}[|X(s)|^p\exp(\lambda s)\exp(-\lambda s)|E_t=s]f_{E_t}(s)ds\\
&=\int_0^\infty \mathbb{E}[|X(s)|^p\exp(\lambda s)|E_t=s]\exp(-\lambda s)f_{E_t}(s)ds\\
&\leq \int_0^\infty C|x_0|^p\exp(-\lambda s)f_{E_t}(s)ds\\
&=C|x_0|^p \mathbb{E}[\exp(-\lambda E_t)].\\
\end{aligned}
\end{equation}
Since $E_t$ is nondecreasing and $E_0=0$, by definition of $E_t$, we claim that $\lim_{t\rightarrow \infty}E_t=\infty$ a.s.. Assume to the contrary that there exists $B>0$ such that $E_t<B$ for all $t>0$ with positive probability, then $D(B)>t$ for all $t>0$  with positive probability. However, by Lemma 12.1 of \cite{bi}, $D(B)$ is bounded, which results in a contradiction. Consequently, $\mathbb{E}[\exp(-\lambda E_t)]\rightarrow 0$ as $t\rightarrow \infty$,
as desired.
\end{proof}

\end{document}